\documentclass[12pt]{amsart}
\usepackage[all]{xy}
\usepackage[T1]{fontenc}

\usepackage{fullpage}
\usepackage{setspace}
\usepackage{amsmath,amsthm,amssymb,  xspace, graphicx, url, amscd, euscript, mathrsfs,epic,eepic,verbatim,stmaryrd}
\usepackage{caption}

\usepackage[usenames,dvipsnames]{color}

\usepackage[colorlinks=true]{hyperref}

\newcommand{\ChDan}[1]{{\color{BrickRed} #1}}

\numberwithin{equation}{section}

1

\allowdisplaybreaks[1]

\newtheorem*{namedtheorem}{\theoremname}
\newcommand{\theoremname}{testing}

\newtheorem{theorem}{Theorem}[section]
\newtheorem{proposition}[theorem]{Proposition}
\newtheorem{proposition-definition}[theorem]
{Proposition-Definition}

\newtheorem{lemma}[theorem]{Lemma}
\newtheorem{conjecture}[theorem]{Conjecture}

\newtheorem{thmintro}{Theorem}

\theoremstyle{definition}

\newtheorem{remark}[theorem]{Remark}

\theoremstyle{remark}

\newcommand{\sD}{\mathscr{D}}
\newcommand{\sE}{\mathscr{E}}

\newcommand{\sX}{\mathscr{X}}
\newcommand{\sY}{\mathscr{Y}}

\newcommand{\defi}[1]{\textsf{#1}} % for defined terms

\newcommand\cA{\mathcal{A}}

\newcommand\cD{\mathcal{D}}

\newcommand\cO{\mathcal{O}}

\newcommand\cT{\mathcal{T}}
\newcommand\cU{\mathcal{U}}

\def\unorm#1{{\underline{#1}}}

\newcommand{\oX}{{\overline{X}}}

\newcommand{\ocA}{{\overline{\mathcal{A}}}}
\newcommand{\tcA}{{\widetilde{\mathcal{A}}}}
\newcommand{\otcA}{{\overline{\widetilde{\mathcal{A}}}}}

\newcommand\uH{\unorm{H}}

\newcommand\uX{\unorm{X}}

\newcommand\GG{\mathbb{G}}

\newcommand\PP{\mathbb{P}}
\newcommand\QQ{\mathbb{Q}}

\newcommand\ZZ{\mathbb{Z}}

\renewcommand\frm{\mathfrak{m}}

\newcommand\frp{\mathfrak{p}}
\newcommand\frq{\mathfrak{q}}

\newcommand{\Gm}{\GG_m}

\newcommand\Spec{\operatorname{Spec}}

\DeclareMathOperator{\Disc}{Disc}

\makeatletter
\let\@wraptoccontribs\wraptoccontribs
\makeatother

\begin{document}

\title[Level structures and Vojta]{Level structures on abelian varieties \\ and Vojta's conjecture}

\author[D. Abramovich]{Dan Abramovich}
\address[Abramovich]{Department of Mathematics, Box 1917, Brown University, Providence, RI, 02912, U.S.A}
\email{abrmovic@math.brown.edu}

\author[A. V\'arilly-Alvarado]{Anthony V\'arilly-Alvarado}
\address[V\'arilly-Alvarado]{Department of Mathematics MS 136, Rice University, 6100 S.\ Main St., Houston, TX 77005, USA}
\email{av15@rice.edu}

\contrib[with an appendix by]{Keerthi Madapusi Pera}
\address[Madapusi Pera]{Department of Mathematics, University of Chicago, 5734 S University Ave, Chicago, IL, USA}
\email{keerthi@math.uchicago.edu}

%14KXX Abelian varieties
%14K10  Algebraic moduli, classification
%14K15  Arithmetic ground fields
%11G18  Arithmetic aspects of modular and Shimura varieties
%11G35  Varieties over global fields

\thanks{Research by D. A. partially supported by NSF grant DMS-1500525. Research by A. V.-A. partially supported by NSF CAREER grant DMS-1352291. Research by K. M. P. is partially supported by NSF grant DMS-1502142.  This paper originated at the workshop ``Explicit methods for modularity of K3 surfaces and other higher weight motives'', held at ICERM in October, 2015. We thank the organizers of the workshop and the staff at ICERM for creating the conditions that sparked this project. We thank {\sc Jordan Ellenberg}, {\sc Aaron Levin}, {\sc Bjorn Poonen}, {\sc Joseph Silverman}, {\sc Doug Ulmer}, and {\sc Yu Yasufuku} for useful discussions, and {\sc Mark Kisin} for comments and for instigating the appendix.}

\date{\today}
\maketitle
\setcounter{tocdepth}{1}

%%%%%%%%%%%%%%%%%%%%%%%%%%%%%%%%%%%%%%%%%%%%%%%%%%%%%%%%%%%%%%%%%%%%%%%%%

	Fix a number field $K$, a prime $p$, and a positive integer $g$. Assuming {\sc Lang}'s conjecture, we showed in~\cite{Alevels} that there exists an integer $r$ such that no principally polarized abelian variety $A/K$ has full level-$p^r$ structure. Recall that, for a positive integer $m$, a \defi{full level}-$m$ structure on an abelian variety $A/K$ is an isomorphism of group schemes on the $m$-torsion subgroup
	\begin{equation}
		\label{eq:levelstructure}
		A[m] \,\xrightarrow{\ \sim\ }\, (\ZZ/m\ZZ)^g \times (\mu_{m})^g.
	\end{equation}
Our goal in this note is to show how to dispose of the dependency on a fixed prime $p$, at the cost of assuming {\sc Vojta}'s conjecture (\cite[Conjecture 2.3]{VojtaABC}, Conjecture~\ref{conj:Vojta} below).

	\begin{thmintro}
		\label{thm:main2}
		Let $K$ be a number field, and let $g$ be a positive integer. Assume {\sc Vojta}'s conjecture. Then there is an integer $m_0$ such that for any $m > m_0$ no principally polarized abelian variety $A/K$ of dimension $g$ has full level-$m$ structure.
	\end{thmintro}

	Theorem~\ref{thm:main2} follows from combining~\cite[Theorem~1.1]{Alevels} and a new result in this note:

	\begin{thmintro}
		\label{thm:main}
		Let $K$ be a number field, and let $g$ be a positive integer. Assume {\sc Vojta}'s conjecture. Then there is an integer $m_0$ such that for any prime $p > m_0$ no principally polarized abelian variety $A/K$ of dimension $g$ has full level-$p$ structure.
	\end{thmintro}
	
%%%%%%%%%%%%%%%%%%%%%%%%%%%%%%%%%%%%%%%%%%%%%%%%%%%%%%%%%%%%%%%%%%%%%%%%%

\tableofcontents

%%%%%%%%%%%%%%%%%%%%%%%%%%%%%%%%%%%%%%%%%%%%%%%%%%%%%%%%%%%%%%%%%%%%%%%%%

\section{Introduction}

%%%%%%%%%%%%%%%%%%%%%%%%%%%%%%%%%%%%%%%%%%%%%%%%%%%%%%%%%%%%%%%%%%%%%%%%%

\subsection{{\textsc{Vojta}}'s conjecture for varieties}

	Before {\sc Merel} proved that torsion on elliptic curves over number fields is uniformly bounded \cite{Merel}, it was known that statements related to {\sc Masser--Oesterl\'e}'s $abc$ conjecture \cite[Conjecture A-B-C]{Frey} or {\sc Szpiro}'s conjecture \cite[Conjecture 1]{Szpiro}  imply such bounds; see {\sc Frey} \cite[Corollary~2.2]{Frey}, {\sc Hindry--Silverman} \cite[Theorem 7.1]{Hindry-Silverman-heights}, {\sc Flexor-Oesterl\'e} \cite{Flexor-Oesterle}. In this paper, we use {\sc Vojta}'s conjecture \cite[Conjecture 2.3]{VojtaABC} as a higher dimensional analogue of the $abc$ conjecture, to study level structures on abelian varieties of dimension $>1$.

	{\sc Vojta}'s conjecture  is a quantitative statement, comparing heights $h_{K_X(D)}(x)$, with respect to the log canonical divisor $K_X(D)$, of rational points $x$ in general position on a {\em projective variety} $X$ over a number field $K$, with the truncated counting function $N^{(1)}_{K}(D,x)$ of such points (see Equation \eqref{eq:N1}) with respect to a normal crossings divisor $D$. The simplest statement, for $K$-rational points, says that if  $K_X(D)$ is big then, for small $\delta$, 
	\[
		N^{(1)}_{K}(D,x) \ \  \geq \ \  (1-\delta)h_{K_X(D)}(x) - O(1)
	\]
for all rational points $x\in X(K)$ outside a Zariski-closed proper subset.

	The general notation is, unfortunately, involved, and explained in \S\ref{Sec:preliminaries}. The conjecture does have qualitative corollaries which are easier to explain. The truncated counting function $N^{(1)}_{K}(D,x)$ measures how often the point $x$ reduces to a point on $D$ modulo primes of $K$. In particular, when $D$ is empty then $N^{(1)}_{K}(D,x)=0$, in which case the statement says that the height $h_{K_X(D)}(x)$ is bounded. Since the height of a big divisor is a counting function outside a Zariski-closed subvariety, this implies that rational points are not Zariski-dense. So {\sc Vojta}'s conjecture implies {\sc Lang}'s conjecture:  the rational points on a positive dimensional variety of general type are not Zariski-dense. More generally, $N^{(1)}_{K}(D,x)=0$ whenever $x$ extends to an integral point on $X \smallsetminus D$. This recovers the statement of the {\sc Lang-Vojta} conjecture: the integral points on a positive dimensional variety of {\em logarithmic}  general type are not Zariski-dense.

	{\sc Campana}  studied varieties where divisors between $K_X$ and $K_X+D$ are big, and for algebraic curves,  stated qualitative conjectures interpolating between {\sc Faltings}' and {\sc Siegel}'s theorems. We will study these intermediate conjectures in higher dimensions in a follow-up note. These statements are qualitative consequences of {\sc Vojta}'s conjecture.

	Our arguments below use {\sc Vojta}'s conjecture for points of bounded degree, which requires an additional  discriminant term $d_K(K(x))$: for small $\delta$ the inequality 
	\[
		N^{(1)}_{K}(D,x) + d_K(K(x))\ \  \geq \ \  h_{K_X(D)}(x)- \delta h_H(x) - O_{[K(x):K]}(1)
	\]
is conjectured to hold,  away from a Zariski-closed subset, where $H$ is a big divisor. Note that when $x \in X(K)$, we have $d_K(K(x)) = 0$.

%%%%%%%%%%%%%%%%%%%%%%%%%%%%%%%%%%%%%%%%%%%%%%%%%%%%%%%%%%%%%%%%%%%%%%%%%

\subsection{{\textsc{Vojta}}'s conjecture for stacks}

	Theorem \ref{thm:main} is decidedly about rational points on {\em stacks}, not varieties. Specifically, an abelian variety $A/K$ corresponds to a rational point on the moduli \emph{stack} $\tcA_g$ of principally polarized abelian varieties. It should thus come as no surprise that, to prove Theorem~\ref{thm:main}, we require a version of {\sc Vojta}'s conjecture for Deligne-Mumford stacks~(Proposition~\ref{prop:VojtaToVojtaDM}), which we deduce from {\sc Vojta}'s original conjecture. 

	Surprisingly, unlike the case of varieties, {\sc Vojta}'s conjecture for stacks requires a discriminant term even for a $K$-rational point: the image of such a point $x$ in $X$ is naturally a stack $\cT_x$ that is in general ramified over the ring of integers $\cO_K$. The corresponding inequality
	\begin{equation}
		\label{eq:Vojtaintro}
		N^{(1)}_{K}(D,x) + d_K(\cT_x)\ \  \geq \ \ h_{K_X(D)}(x)- \delta h_H(x) - O(1)
	\end{equation}
holds away from a Zariski-closed proper subset, conditional on {\sc Vojta}'s conjecture for varieties.

	Proposition~\ref{prop:VojtaToVojtaDM} is proved by passing to a branched covering $Y \to X$ by a variety. Such a covering was constructed by {\sc Kresch} and {\sc Vistoli} in \cite[Theorem 1]{KreschVistoli}; we adapt their construction to stacks with normal crossings divisors in Proposition \ref{lem:KV}.

	While the discriminant term comes naturally from {\sc Vojta}'s statement for points of bounded degree,  one might contemplate doing away with it. It is, however, indispensable, at least if one is to state a conjecture that is not patently false. Consider the root stack $X = \PP^2(\sqrt C)$, where $C$ is a curve of degree $\geq 7$, and let $D=\emptyset$. Then $K_X$ is ample, $N^{(1)}_{K}(D,x) \equiv 0$, and yet there is a dense collection of rational points on the open subset $\PP^2 \smallsetminus C \subset X$.

%%%%%%%%%%%%%%%%%%%%%%%%%%%%%%%%%%%%%%%%%%%%%%%%%%%%%%%%%%%%%%%%%%%%%%%%%

\subsection{Abelian varieties, counting functions and discriminants}

	Fix an integer $m_0$ and consider the set $\tcA_g(K)_{p\geq m_0}$ of points corresponding to abelian varieties admitting full level-$p$ structures, for primes $p\geq m_0$. Our task is to show that for large $m_0$ this set is empty. To this end, it is natural to focus on an irreducible component  $X \subset \overline{\tcA_g(K)_{p\geq m_0}}$ of the Zariski-closure. This leads to the following setup: Consider a closed substack $X \subset \tcA_g$, a resolution of singularities $X' \to X$, and a normal crossings compactification $\oX'$ with boundary divisor $D$. Following {\sc Zuo} \cite[Theorem 0.1(ii)]{Zuo}, we showed  in \cite[Theorem 1.7]{Alevels} that $K_{\oX'}(D)$ is big.

	With a version of {\sc Vojta}'s conjecture for stacks in hand, the key to proving Theorem \ref{thm:main} is to show that,  for points $x\in X'(K)$ corresponding to abelian varieties with full level-$p$ structure, the terms $N^{(1)}_{K}(D,x)$ and $d_K(\cT_x)$ on the left hand side of~\eqref{eq:Vojtaintro} are small compared to the height $h_{K_{\oX'}(D)}(x)$, as soon as $p$ is large enough. 

	To this end, we show that each one of these terms is bounded by a small fraction of the height $h_D(x)$; see Lemmas \ref{lem:primelevels} and \ref{lem:primelevels-truncated}. First, to bound the truncated counting function $N^{(1)}_{K}(D,x)$ we use the fact that the compactified moduli space $\otcA_g^{[p]}$ of abelian varieties with full level-$p$ structure is highly ramified over the compactification $\otcA_g$ along the boundary. This is well-known away from characteristic $p$; see  \cite[Proposition 4.1]{Alevels}. The remaining case of characteristic $p$ is proven in Proposition \ref{prop:full level structure} as part of the Appendix by Keerthi {\sc Madapusi Pera}, where the structure of the boundary is  described using {\sc Mumford}'s construction.

	Second, using standard discriminant bounds, we show the discriminant term $d_k(\cT_x)$ grows at most like $\log p$. Meanwhile, the height $h_D(x)$ grows at least linearly in $p$. For this we use a point-counting argument  of {\sc Flexor--Oesterl\'e} \cite[Th\'eor\`eme  3]{Flexor-Oesterle} and {\sc Silverberg} \cite[Theorem 3.3]{Silverberg-finite-order} (see also  {\sc Kamienny}  \cite[\S6(2a)]{Kamienny-root-p}) to show that $x$ reduces to $D$ modulo a small prime, whose contribution to  $h_D(x)$ is at least proportional to $p$, since  $\otcA_g^{[p]}\to \otcA_g$ is highly ramified. 

	Together, these two bounds can be leveraged to show that the totality of points $x \in X'(K)$ corresponding to abelian varieties over $K$ with full level-$p$ structure for $p \gg 0$ is contained in a Zariski-closed proper subset of $X'$. A Noetherian induction argument allows us to deduce Theorem~\ref{thm:main} from this result.

%%%%%%%%%%%%%%%%%%%%%%%%%%%%%%%%%%%%%%%%%%%%%%%%%%%%%%%%%%%%%%%%%%%%%%%%%

\section{Preliminaries}
	\label{Sec:preliminaries}
	\addtocontents{toc}
	{\hspace{0.24in} We review, and extend to algebraic stacks, {\sc Vojta}'s notation of relative

	\hspace{0.26in} discriminants $D_K(E)$, multiplicities $n_q(\cD,x)$, and truncated counting 

	\hspace{0.26in} functions $N^{(1)}_K(D,x)$. We construct coverings of stacks by schemes

	\hspace{0.26in} (Proposition \ref{lem:KV}).
	}

	In this section, we set up notation that will remain in force throughout. Let $K$ be a number field, and let $\overline{K}$ be a fixed algebraic closure of $K$. We write $\cO_K$ for the ring of integers of $K$, and $\Disc(\cO_K)$ for its discriminant. We denote by $M^0_K$ the set of nonzero primes of $\cO_K$; for $p \in M^0_K$, we write $\cO_{K,p}$ for the localization of $\cO_K$ at $p$ and $\kappa(p)$ for the residue field. We use $S$ to denote a finite set of places of $K$ that includes the infinite places, and $\cO_{K,S}$ for the ring of $S$-integers of $K$.

	For a finite extension $L/K$, we write $\Omega_{\cO_L/\cO_K}$ for the the module of K\"ahler differentials.

%%%%%%%%%%%%%%%%%%%%%%%%%%%%%%%%%%%%%%%%%%%%%%%%%%%%%%%%%%%%%%%%%%%%%%%%%
	
\subsection{Discriminants of fields}
	\label{Sec:discriminants-fields}

	For a finite extension $E/K$, following {\sc Vojta}, define the \defi{relative logarithmic discriminant} as
	\[
		d_K(E) = \frac{1}{[E:K]}\log |\Disc(\cO_E)| \, - \, \log |\Disc(\cO_K)|.
	\]
Noting that $(\Disc(\cO_K)) = N_{K/\QQ}\det \Omega_{\cO_K/\ZZ}$ as ideals, we have 
	\[
		d_K(E) = \frac{1}{[E:K]}\deg \Omega_{\cO_E/\cO_K};
	\] 
see \cite[Page 1106]{VojtaABC}. The right hand side can be decomposed into a sum of local contributions
	\[
		\deg \Omega_{\cO_E/\cO_K} =  \sum_{\frp\in M^0_E}\deg_\frp \Omega_{\cO_E/\cO_K} = \sum_{\frp\in M^0_E}\text{length}(\Omega_{\cO_{E_p}/\cO_{K_p}})\log |\kappa(\frp)|.
	\] 
For $p\in M^0_K$ we write 
	\[
		d_K(E)_p :=  \frac{1}{[E:K]}  \sum_{\frp\mid p}\deg_\frp \Omega_{\cO_E/\cO_K}
	\]
for the contribution of the primes above $p$, so that $d_K(E) = \sum _{p\in M^0_K}  d_K(E)_p$. 

	If $L/E$ is a further finite extension, the formula for discriminants in the tower $L/E/K$ gives
	\begin{align}
		\label{Eq:discr-global} d_K(L) &=\frac{1}{[E:K]} d_E(L) + d_K(E) \ \ \, =  \frac{1}{[L:K]} \deg\Omega_{\cO_L/\cO_E} + d_K(E),\\
		\label{Eq:discr-local} d_K(L)_p &=\frac{1}{[E:K]} d_E(L)_p + d_K(E)_p =  \frac{1}{[L:K]} \deg \Omega_{\cO_{L,p}/\cO_{E,p}} + d_K(E)_p.
	\end{align}
In particular, if $L/E$ is unramified above $p \in M^0_K$ then $d_K(L)_p  = d_K(E)_p.$ 

%%%%%%%%%%%%%%%%%%%%%%%%%%%%%%%%%%%%%%%%%%%%%%%%%%%%%%%%%%%%%%%%%%%%%%%%%

\subsection{Discriminants of stacks}
	\label{Sec:discriminants}

	We shall need analogous definitions where $\Spec \cO_{E}$ is replaced by a normal separated Deligne--Mumford stack $\cT$ with coarse moduli scheme $\Spec \cO_{E}$: 
	\begin{align*}
		d_K(\cT) &= \frac{1}{\deg (\cT/\cO_K)} \deg(\Omega_\cT) \, - \,  \log |\Disc(\cO_K)| = \frac{1}{\deg (\cT/\cO_K)} \deg(\Omega_{\cT/\Spec \cO_K}), \\
		d_K(\cT)_p &= \frac{1}{\deg (\cT/\cO_K)} \deg(\Omega_{\cT_p}) \, - \, \log |\Disc(\cO_K)| = \frac{1}{\deg (\cT/ \cO_{K})}\deg(\Omega_{\cT_p/\Spec \cO_{K,p}}).
	\end{align*} 
The quantity $\deg (\cT/ \cO_K)$ is in general rational, as the fiber over $\Spec K$ might be a gerbe over $\Spec E$\footnote{We can redefine $\cT_x$ to be the normalization in the field $E$, so that $\deg(\cT/\cO_K) = [E:K]$, an integer.}. However, we still have $d_K(\cT) = \sum_{p\in M^0_K} d_K(\cT)_p$.

	Choose a morphism $\Spec \cO_F \to \cT$  unramified above $p$, so that $(\Omega_{\cT/\Spec \cO_K})_{\Spec \cO_{F,p}} =  \Omega_{\cO_{F,p}/\cO_{K,p}}.$ Since $[F:K] =  \deg (\cT/\cO_{K}) \cdot \deg (\Spec \cO_F/\cT)$, we can compute $d_K(\cT)_p $ entirely with schemes:

	\begin{lemma} 
		For $\Spec \cO_F \to \cT$  unramified above $p$, we have $d_K(\cT)_p  = d_K(F)_p.$
		\qed
	\end{lemma}

	We deduce  analogues of  Equations \eqref{Eq:discr-global}  and \eqref{Eq:discr-local}:

	\begin{lemma}
		\label{Lem:discriminant}
		Let $L/E$ be a finite extension field and $\pi\colon\Spec \cO_L \to \cT$ a morphism. Then 
		\begin{align*}
			d_K(L)_p &= \frac{1}{[L:K]} \deg (\Omega_{\Spec (\cO_{L,p}) /\cT_p}) + d_K(\cT)_p\\ 				\intertext{and}
			d_K(L) &= \frac{1}{[L:K]} \deg (\Omega_{\Spec (\cO_{L}) /\cT}) + d_K(\cT).
		\end{align*}
	\end{lemma}	

	\begin{proof}
		To prove the local statement, choose $\psi\colon\Spec \cO_F \to \cT$  unramified above $p$, and let $\cU = \Spec (\cO_{F}) \times_\cT \Spec (\cO_{L})$ with projection $\phi\colon \cU \to  \Spec (\cO_{L})$. These objects fit together in the  commutative diagram
		\[
			\xymatrix{\cU \ar[r]^{\phi} \ar[d] & \Spec \cO_L \ar[d]^\pi & \\
			\Spec \cO_F \ar[r]^\psi & \cT \ar[r]^\tau & \Spec\cO_K}
		\]
		We have 
		\begin{align*}
			\Omega_{\cU_p/ \Spec (\cO_{F,p})} &=  \Omega_{\cU_p/ \cT_p} = \phi^* \Omega_{\Spec (\cO_{L,p}) /\cT_p}, \\
			\Omega_{\cU_p/ \Spec (\cO_{K,p})} &=  \phi^* \Omega_{\cO_{L,p} / \cO_{K,p}}, 
		\end{align*} 
and 
		\[
			\Omega_{\cO_{F,p} / \cO_{K,p}} =  \psi^*\Omega_{\cT_p / \Spec (\cO_{K,p})}
		\]
The projection formula gives
		\begin{align*} 
			\deg (\Omega_{\cU_p/ \cT_p}) &= \deg(\psi)  \deg  (\Omega_{\Spec (\cO_{L,p}) /\cT_p})\\
			\deg (\Omega_{ \cO_{F,p} / \cO_{K,p}})  &= \deg(\psi) \deg (\Omega_{\cT_p /\Spec (\cO_{K,p})})
		\end{align*}
and finally 
		\begin{align*} 
			d_K(L)_p 
 			&= \frac{1}{[L:K]} \deg (\Omega_{ \cO_{L,p} /  \cO_{K,p}}) 
 			= \frac{1}{[L:K] \deg \psi} \deg (\Omega_{\cU_p / \Spec (\cO_{K,p})}) \\ 
 			&= \frac{1}{[L:K] \deg \psi} \left( \deg (\Omega_{\cU_p / \Spec (\cO_{F,p})}) +  (\deg \pi) \deg 			(\Omega_{ \cO_{F,p} /  \cO_{K,p}})\right) \\
			&= \frac{1}{[L:K] \deg \psi}  \deg (\Omega_{\cU_p / \cT_p}) + \frac{1}{  \deg(\cT/\Spec \cO_K)} \deg (\Omega_{\cT_p / \Spec (\cO_{K,p})}) \\
 			& = \frac{1}{[L:K]} \deg (\Omega_{\Spec (\cO_{L,p}) /\cT_p}) + d_K(\cT)_p,
 		\end{align*}
as required. The global formula follows by summing over $p\in M^0_K$.
	\end{proof}

%%%%%%%%%%%%%%%%%%%%%%%%%%%%%%%%%%%%%%%%%%%%%%%%%%%%%%%%%%%%%%%%%%%%%%%%%

\subsection{Heights on stacks}

	For a divisor $H$ on a smooth projective scheme $Y$, we denote by $h_H(x)$ the Weil height of $x$ with respect to $H$, which is well-defined up to a bounded function on $Y(\overline{K})$. To define a notion of height on a Deligne-Mumford stack, we pull back to a cover by a scheme and work there instead.  Let $X/K$ be a smooth proper Deligne-Mumford stack with projective coarse moduli scheme and let $H \subset X$ be a divisor. Let $f\colon Y\to X$ be the finite flat surjective morphism from a smooth projective scheme $Y$ guaranteed by~\cite[Theorem~1]{KreschVistoli} or~Proposition~\ref{lem:KV} below.  For a point $x \in X(\overline{K})$, let $y \in Y(\overline{K})$ be a point over $x$, and define
\[
h_H(x) := h_{f^*(H)}(y).
\]
This definition has the advantage of having functoriality properties of heights built into it. It is also compatible with passing to the coarse moduli space, at the price of working with $\QQ$-Cartier divisors on slightly singular schemes: any divisor $H$ on $X$ is the pullback of a $\QQ$-Cartier divisor $\uH$ on the coarse moduli space ${\underline X}$, and if $\underline x \in \uX$ is the image of $x$ then 
\[
h_H(x) = h_{\underline H}({\underline x}).
\]  

	Our definition has the disadvantage that there can be infinitely many rational points (with the same image in $\uX$) with the same height.  In a forthcoming paper, {\sc Ellenberg, Satriano} and {\sc Zureick-Brown} construct an alternative notion of height on a stack, with the property that there are only finitely many non-isomorphic points with bounded height.

%%%%%%%%%%%%%%%%%%%%%%%%%%%%%%%%%%%%%%%%%%%%%%%%%%%%%%%%%%%%%%%%%%%%%%%%%

\subsection{Normal crossings models}
	\label{ss:ncms}

	Let $(\sX,\sD)$ be a pair with $\sX \to \Spec \cO_{K,S}$ a smooth proper morphism from a scheme or Deligne-Mumford stack, and $\sD$ a fiber-wise normal crossings divisor on $\sX$. Let $(X,D)$ be the generic fiber of the pair $(\sX,\sD)$; we say that $(\sX,\sD)$ is a \defi{normal crossings model} of the pair $(X,D)$. Write $\sD = \sum_i \sD_i$ and let $D_i$ be the generic fiber of $\sD_i$. 
	
%%%%%%%%%%%%%%%%%%%%%%%%%%%%%%%%%%%%%%%%%%%%%%%%%%%%%%%%%%%%%%%%%%%%%%%%%

\subsection{Intersection multiplicities on schemes and stacks}

	For $R$ an integral extension of $\cO_{K,S}$, and $q \subset R$ a nonzero prime ideal, let $R_q$ be the localization of $R$ at $q$, with maximal ideal $\frm_q$ and residue field $\kappa(q)$. 

	We first define multiplicities for integral points. Let $x \in \sX(R_q)$, and define $n_q(\sD_i,x)$ as the intersection multiplicity of $x$ and $\sD_i$. In other words, letting $I_{\sD_i}$ denote the ideal of $\sD_i$, we have an equality of ideals in $R_q$
	\[
		I_{\sD_i}\big|_x = \frm_q^{n_q(\sD_i,x)}.
	\]
Note that if $R'$ is an integral extension of $R$, with maximal ideal $\frq\mid q$ and if $y\in \sX(R'_\frq)$ is the composite of $\Spec R'_\frq \to \Spec R_q \to \sX$, then we have 
$	n_\frq(\sD_i,y) = e(\frq\mid q) n_q(\sD_i,x)$, where $e(\frq\mid q) $ is the ramification index of $\frq$ over $q$. 

	This observation prompts the following extension of the definition of $n_q(\sD_i,x)$ to a rational point $x$ of  $\sX$.
Denoting by $K(R)$ and $K(R')$ the respective fraction fields of $R$ and $R'$, if $x \in \sX(K(R))$ and if $y\in \sX(R')$ is an integral point over $x$, then the quantity 
	\begin{equation}
		\label{Eq:n_q-extension} 
		 n_q(\sD_i,x) := 	\frac{1}{e(\frq\mid q)} n_\frq(\sD_i,y)
	\end{equation}
is well-defined.	 

	Finally, define $n_q(\sum a_i\sD_i,x) := \sum_i a_in_q(\sD_i,x)$.

%%%%%%%%%%%%%%%%%%%%%%%%%%%%%%%%%%%%%%%%%%%%%%%%%%%%%%%%%%%%%%%%%%%%%%%%%

\subsection{Counting functions}

	Following {\sc Vojta}~\cite[p.~1106]{VojtaABC}, for $x \in \sX(\overline K)$, with residue field $K(x)$, define the \defi{truncated counting function}
	\begin{equation}
		\label{eq:N1}
		N^{(1)}_{K}(D,x) = \frac{1}{[K(x):K]}\sum_{\substack{q \in \Spec\cO_{K(x),S} \\ n_q(\sD,x)>0}} 
		\log |\kappa(q)|.
	\end{equation}
The quantity on the right hand side of~\eqref{eq:N1} depends on the model $(\sX,\sD)$ and the finite set $S$ only up to  a bounded function on $X(\overline{K})$. However, we are interested in this quantity only up to such functions. Hence, the notation $N^{(1)}_{K}(D,x)$ does not reflect the model $(\sX,\sD)$ or the finite set $S$.

	By~\cite[p.~1113]{VojtaABC} or~\cite[Theorem~B.8.1(e)]{HindrySilverman} we have the bound

	\begin{equation}
		\label{Eq:counting-height}
		N^{(1)}_{K}(D,x) \leq 
		 \frac{1}{[K(x):K]}\sum_{\frq} n_\frq(D,x) \log |\kappa(\frq)| \leq 
		h_{D}(x) + O(1) 
	\end{equation}
which can be further improved whenever we bound the multiplicities $n_\frq(D,x)$ from below. 

%%%%%%%%%%%%%%%%%%%%%%%%%%%%%%%%%%%%%%%%%%%%%%%%%%%%%%%%%%%%%%%%%%%%%%%%%

\subsection{Coverings of stacks}

%%%%%%%%%%%%%%%%%%%%%%%%%%%%%%%%%%%%%%%%%%%%%%%%%%%%%%%%%%%%%%%%%%%%%%%%%

	We require the following version of~\cite[Theorem~1]{KreschVistoli}, due to {\sc Kresch} and {\sc Vistoli}, adapted to the case of a stack with a normal crossings divisor.

	\begin{proposition}
		\label{lem:KV}
		Suppose  $X/K$ is a smooth proper Deligne--Mumford stack with projective moduli scheme, with a normal crossings divisor $D\subset X$.  Then there exists a finite surjective morphism $\pi\colon Y \to X$ such that $Y$ is a smooth projective irreducible scheme, $D_Y := \pi^*D \subset Y$ is a normal crossings divisor, and the ramification divisor $R$ of $Y \to X$ meets every stratum of $D_Y$ properly.
	\end{proposition}

	The proof of this proposition requires the following slicing lemma.

	\begin{lemma}[{See \cite[Lemma 1]{KreschVistoli}}]
		\label{lem:KVlem1}
		Let $f\colon U \to V$ be a morphism of quasi-projective varieties over an infinite field, with constant fiber dimension $r>0$; Let $D \subset V$ be a divisor. Assume $U$ is smooth and  $D_U = f^{-1} D$ is a simple normal crossings divisor. Let $U \subset \PP^N$ be a projective embedding. Denote by $D_U^I$ the closed strata of $(U,D_U)$, and assume further that $D_U^I \to D^I:= f(D_U^I)$ is generically smooth for each $I$. Then for sufficiently high $d$, the intersection $D_U^I \cap H^{(d)}$ of each stratum $D_U^I$ with a general hypersurface $H^{(d)}\subset\PP^N$ of degree $d$ is a smooth Cartier divisor in $D_U^I$, generically smooth and of constant fiber dimension $r-1$ over $D^I$.
	\end{lemma}

	\begin{proof}[Proof of the Lemma]
		For each $I$, \cite[Lemma 1]{KreschVistoli} applied to $U \to V$ replaced by $D_U^I \to D^I$, provides an an integer $d_I$ and, for each $d\geq d_I$, an open subset of $\Gamma(\PP^N, \cO(d))$ where $D_U^I \cap H^{(d)}$ is a smooth Cartier divisor in $D_U^I$ of constant fiber dimension $r-1$ over $D^I$. By {\sc Bertini}'s Theorem, after possibly enlarging $d_I$ we may replace it by a smaller  open subset where $D_U^I \cap H^{(d)} \to D^I$ is also generically smooth. Take  $d\geq\max_I d_I$.  
	\end{proof}

	\begin{proof}[Proof of Proposition~\ref{lem:KV}]
		First, we note that $X$ is a quotient stack: to see this, one combines~\cite[Theorem~2]{KreschVistoli} together with a result of {\sc Gabber} implying that the Azumaya Brauer group of a quasi-projective scheme over a field coincides with the cohomological Brauer group~\cite{deJong-Gabber}. Next, proceeding as in the proof of~\cite[Theorem~1]{KreschVistoli}, one can construct a smooth projective morphism of stacks $\pi\colon P \to X$ with a representable open substack $Q\subset P$, whose fiber dimension is greater than that of $P\smallsetminus Q$. The induced morphism on coarse moduli spaces $U \to \underline{X}$ is proper, and $U$ is quasi-projective by~\cite[Lemma~2]{KreschVistoli}.

		Beginning with the map $U \to \underline{X}$ and the image of $D$ via $X \to \underline{X}$, repeated applications of Lemma~\ref{lem:KVlem1} yield a closed subscheme $Y\subset U$ such that the map $Y \to \underline{X}$ is finite and surjective, and such that $D_Y$ is a normal crossings divisor whose strata meet the ramification divisor of $Y \to \underline{X}$ properly. We can assume that $Y$ is disjoint from the image of $P\smallsetminus Q$ in $U$, by dimension reasons. We can thus lift $Y$ to a representable substack of $Q$, because $Q$ is representable, and get the desired morphism $Y \to X$. 
	\end{proof}

%%%%%%%%%%%%%%%%%%%%%%%%%%%%%%%%%%%%%%%%%%%%%%%%%%%%%%%%%%%%%%%%%%%%%%%%%

\subsection{Rational and integral points on stacks}

	We will make use of the following standard observations:

	\begin{lemma} 
		\label{Lem:lifting-points}
		 Let $R$ be a dedekind domain with fraction field $K$.
		 \begin{enumerate}
			\item Let $f\colon Y \to X$ be a proper representable morphism of algebraic stacks over $R$. Let $y\in Y(K)$ and $x = f(y)$. Then $y$ extends to a point  $\eta \in Y(R)$ if and only if $x$ extends to a point  $\xi \in X(R)$. 
			\item Let $X/R$ be an algebraic  stack, $Y/R$ a proper scheme, and $Y \to X$ a morphism.  If $x\in X(K)$ is the image of  $y\in Y(K)$ then it extends to $\xi\in X(R)$. 
			\item  Let $X/R$ be a proper algebraic  stack, $Y/R$ a proper scheme, and $Y \to X$ a flat surjective morphism of degree $M$. Let $x \in X(K)$. There is a finite extension $L/K$, with $[L:K]\leq M$ and $R_L\subset L$ the integral closure of $R$, and a point $\xi \in X(R_L)$ lifting $x$.
		\end{enumerate}
	\end{lemma}
	
	\begin{proof} \ 
		\begin{enumerate}
			\item Given $\eta \in Y(R)$ we have $f(\eta) = \xi \in X(R)$. If $\xi \in X(R)$ consider the fibered product $Z = \Spec R \times_X Y$ defined by $\xi$, which is representable and proper over $R$. Then $y$ gives a point of $Z(K)$, which extends to $R$ by the valuative criterion of properness.
			\item By the valuative criterion for properness $y$ extends to $\eta\in Y(R)$, whose composition with $Y \to X$ gives $\xi\in X(R)$.
			\item The  $K$-scheme $Z = \Spec K \times_XY$ is finite  of degree $M$, hence admits a rational point $y\in Z(L)$ with $[L:K]\leq M$. The composition $\Spec L \to Z \to Y$ extends to $\eta \in Y(R_L)$ by the valuative criterion for properness, and its composition with $f$ is a point $\xi \in X(R_L)$ lifting $x$.
		\end{enumerate}
	\end{proof}
	
%%%%%%%%%%%%%%%%%%%%%%%%%%%%%%%%%%%%%%%%%%%%%%%%%%%%%%%%%%%%%%%%%%%%%%%%%

\section{Vojta's conjecture for varieties and stacks}
	\addtocontents{toc}{\hspace{0.24in} We recall {\sc Vojta}'s conjecture (Conjecture \ref{conj:Vojta}) and show that it implies 

	\hspace{0.26in} a version for algebraic stacks (Proposition \ref{prop:VojtaToVojtaDM}).}

	We write $K_X$ for the canonical divisor class of a smooth variety or smooth Deligne--Mumford stack $X$.

	\begin{conjecture}[{\sc Vojta} {\cite[Conjecture 2.3]{VojtaABC}}]
		\label{conj:Vojta}
		Let $X$ be a smooth projective variety over a number field $K$, $D$ a normal crossings divisor on $X$, and $H$ a big line bundle on $X$.  Let $r$ be a positive integer and fix $\delta>0$. Then there is a proper Zariski-closed subset $Z \subset X$ containing $D$ such that
		\[
			N^{(1)}_{K}(D,x) + d_K(K(x))\ \  \geq \ \  h_{K_X(D)}(x)- \delta h_H(x) - O(1)
		\]
for all $x\in X(\overline{K})\smallsetminus Z(\overline{K})$ with $[K(x):K]\leq r$.
	\end{conjecture}
	
	We note that variants of the conjecture above have been stated, involving the counting function $N_{K}(D,x)= \frac{1}{[K(x):K]}\sum_{\frq} n_\frq(D,x) \log |\kappa(\frq)|$ and a different coefficient in front of the discriminant term $d_K(K(x))$. It may be possible to deduce results similar to Theorem \ref{thm:main} from these variants; we do not do so here.
	
	We shall need a version of {\sc Vojta}'s conjecture for Deligne-Mumford stacks.  For a smooth proper Deligne-Mumford stack $\sX \to \Spec \cO_{K,S}$ we write $X = \sX_K$ for the generic fiber, which we assume is irreducible, and $\underline{X}$ for the coarse moduli space of $X$. Similarly, for a normal crossings divisor $\sD$ of $\sX$, we write $D$ for its generic fiber.

	Given a  point $x \in \sX(\overline K)$, we take the Zariski closure and normalization of its image, and extend it uniquely to a morphism, denoted $\cT_x \to \sX$, where $\cT_x$ is a normal stack with coarse moduli scheme $\Spec \cO_{K(x),S}$. We thus have the relative discriminant $d_K(\cT_x)$ defined in \S\ref{Sec:discriminants}.

	\begin{proposition}[{\sc Vojta} for stacks]
		\label{prop:VojtaToVojtaDM}
		Assume {\sc Vojta}'s Conjecture~\ref{conj:Vojta} holds. 
Let $\sX\to \Spec \cO_{K,S}$, $X$, $\underline{X}$, and $D$ be as above. Suppose that $\underline{X}$ is projective, and let $H$ be a big line bundle on it. Let $r$ be a positive integer and fix $\delta>0$. Then there is a proper Zariski-closed subset $Z \subset X$ containing $D$ such that 
		\[
			N^{(1)}_{K}(D,x) + d_K(\cT_x)\ \  \geq \ \ h_{K_X(D)}(x)- \delta h_H(x) - O(1)
		\]
for all $x\in X(\overline K)\smallsetminus Z(\overline{K})$ with $[K(x):K]\leq r$.
	\end{proposition}

	\begin{proof}
		Let $Y \to X$ be the finite cover of $X$ guaranteed by Proposition~\ref{lem:KV}. Possibly after enlarging $S$ we may assume $Y\to X$ extends to $\pi\colon\sY \to \sX$ for some model $\sY$ of $Y$, so  a point  $y\in Y(K(y))$ extends to $\Spec \cO_{K(y),S} \to \sY$, and composes to $\Spec \cO_{K(y),S} \to \sX$. We denote $\pi(y)=x$, and its extension as a stack by $\cT:= \cT_x \to \sX$.

		By Riemann-Hurwitz, we have
		\begin{align*} 
			K_Y+D_Y &= (\pi^* K_X + R) + \pi^* D = \pi^*(K_X+D) + R.  \\
			\intertext{  Thus for $y \in Y(\overline{K})$ with $\pi(y) = x$ outside a proper Zariski-closed subset of $Y$, we have}
			h_{K_Y(D_Y)}(y) &= h_{K_X(D)}(x) + h_R(y) + O(1).
			\intertext{Let $\underline{\pi}\colon Y \to \uX$ be the composition of $\pi$ with the natural map $X \to \uX$. Let $B = \underline{\pi}^*(H)$; then $B$ is big, and by functoriality of heights, we have} 
			h_{B}(y) &= h_{H}(x) + O(1).
		\end{align*} 
for all $y \in Y(\overline{K})$. Let $\sD_Y = \pi^*\sD$.

		\begin{lemma}
			\label{Lem:compare-N1}
			$N^{(1)}_K(D_Y,y) \leq N^{(1)}_K(D,x).$
		\end{lemma}

		\begin{proof}
			Note that $n_q(\sD_Y,y) >0$ if and only if $n_q(\sD,x) >0$. Then
			\begin{align*}
				N^{(1)}_K(D_Y,y) &=  \frac{1}{[K(y):K]}\sum_{\substack{\frq \in \Spec\cO_{K(y),S} \\ n_{\frq}(\sD_Y,y) >0}}  \log |\kappa(\frq)|.\\
				&= \frac{1}{[K(y):K]}\sum_{q \,:\, n_q(\sD,x)>0}\sum_{\frq \,\mid\, q} \log |\kappa(\frq)| \\
				&\leq \frac{1}{[K(y):K]}\sum_{q \,:\, n_q(\sD,x) >0}\sum_{\frq \,\mid\, q} e(\frq\mid q) \log |\kappa(\frq)| \\
				&= \frac{1}{[K(y):K]}\sum_{q \,:\, n_q(\sD,x) >0} [K(y) : K(x)] \log |\kappa(q)| \\
				&= \frac{1}{[K(x):K]}\sum_{q \,:\, n_q(\sD,x) >0} \log |\kappa(q)| \ \ =\ \  N^{(1)}_K(D,x)
			\end{align*}
		\end{proof}

		\begin{lemma} 
			\label{Lem:disc-height}
			$\frac{1}{[K(y):K]} \deg_y \Omega_{\cO_{K(y)}/\cO_{\cT}} \ \ \leq \ \  h_R(y) \ +  \ O(1)$.
		\end{lemma}

		\begin{proof}
			Write $\sY_\cT = \sY \times_\sX \cT$. The morphism $\cT \to \sX$ is representable since it is the normalization of a substack. It follows that $\sY_\cT$ is a scheme. Also $\Spec \cO_{K(y)} \to \sY_\cT$ is the normalization of the image subscheme $\overline{\text{Im}(y)}$.

			Therefore
			\begin{align*}
				\deg_y \Omega_{ \cO_{K(y)} /\cT} &\ \  \leq \ \ \deg_y \Omega_{\text{Im}(y) /\cT}\ \  \leq \ \ \deg_y \Omega_{ \sY_\cT/\cT}\ \  \leq \ \ \deg_y\Omega_{\sY/\sX} \\
				\intertext{(since $\deg \Omega$ drops when passing to normalization, subscheme or pullback)}
				&\ \ = \ \ \deg_y \det \Omega_{\sY/\sX} \ \  = \ \  [K(y):K]\ \cdot \ h_R(y)\ \ + \ \ O(1)
			\end{align*}
as needed.
		\end{proof}

		Continuing with the proof of Proposition~\ref{prop:VojtaToVojtaDM}, Conjecture~\ref{conj:Vojta} for $Y$ gives
		\begin{equation}
			\label{eq:vojtaapplied}
			N^{(1)}_K(D_Y,y) + d_K(K(y)) \ \ \geq\ \  h_{K_Y+D_Y}(y) - \delta   h_{B}(y) + O_{[K(y):K(x)]}(1).
		\end{equation}
for $y$ away from a proper closed subset.  By Lemma \ref{Lem:discriminant} we have $$d_K(K(y)) = \frac{1}{[K(y):K]} \deg_y \Omega_{\cO_{K(y)}/\cO_{\cT}} + d_K(\cT).$$ By Lemmas \ref{Lem:compare-N1} and \ref{Lem:disc-height} the left hand side of~\eqref{eq:vojtaapplied} is majorized by 
		\[
			N^{(1)}_K(D,\pi(y)) + h_R(y) + d_K(\cT)  + O_{[K(y):K(x)]}(1). 
		\]
On the other hand, for the right hand side of~\eqref{eq:vojtaapplied}, we have 
		\[
			h_{K_Y(D_Y)}(y) - \delta   h_{B}(y) = h_{K_X(D)}(x) + h_R(y)  - \delta h_{H}(x) + O(1).
		\]
All together, we obtain
		\[
			N^{(1)}_K(D,x) + h_R(y) + d_K(\cT) \geq h_{K_X(D)}(x) + h_R(y)  - \delta h_{H}(x) + O_{[K(y):K(x)]}(1),
		\]
which, after canceling $h_R(y)$, gives
		\[
			N^{(1)}_K(D,x) + d_K(\cT) \geq h_{K_X(D)}(x) - \delta h_{H}(x) + O_{[K(y):K(x)]}(1).
		\]
A point $x$ with $[K(x):K]\leq r$ is the image of a point $y$ with $[K(y):K]\leq r\cdot \deg\pi$. Thus, the Proposition for $\sX$, $\sD$, $H$, $r$ and $\delta$ follows from Conjecture~\ref{conj:Vojta} applied to $Y$, $\pi^*\sD$, $B$, $r\cdot \deg\pi$, and $\delta$.
	\end{proof}

%%%%%%%%%%%%%%%%%%%%%%%%%%%%%%%%%%%%%%%%%%%%%%%%%%%%%%%%%%%%%%%%%%%%%%%%%

\section{Proof of the main result}
	\addtocontents{toc}{\hspace{0.24in} We show that, for an abelian variety with full level-$p$ structure,  the discriminant 

	\hspace{0.26in} term $d_K(\cT_x)$ and  the truncated counting function $N_K^{(1)}(D,x)$ are majorized

	\hspace{0.26in}   by $h_{\epsilon D}(x)$ (Lemmas \ref{lem:primelevels} and \ref{lem:primelevels-truncated}). 
	We use this to prove Theorem \ref{thm:main}}

%%%%%%%%%%%%%%%%%%%%%%%%%%%%%%%%%%%%%%%%%%%%%%%%%%%%%%%%%%%%%%%%%%%%%%%%%

\subsection{Moduli spaces and toroidal compactifications}

	We follow the notation of~\cite{Alevels}. However, we work over $\Spec \ZZ$:
	
	\medskip
	
	\begin{tabular}{ll}
		$\tcA_g \subset \otcA_g$ & a toroidal compactification of the moduli \emph{stack} of \\
	& principally polarized abelian varieties of dimension $g$\\[1mm]
		$\cA_g \subset \ocA_g$ & the resulting compactification of the moduli \emph{space} of \\
	& principally polarized abelian varieties of dimension $g$\\[1mm]
		$\tcA_g^{[m]} \subset \otcA_g^{[m]}$ & a compatible toroidal compactification of the moduli \emph{stack} of \\
	& principally polarized abelian varieties of dimension $g$\\	
	& with full level-$m$ structure\\[1mm]
		$\cA_g^{[m]} \subset \ocA_g^{[m]}$ & the resulting compactification of the moduli space of \\
	& principally polarized abelian varieties of dimension $g$\\
	& with full level-$m$ structure\\[1mm]
	\end{tabular}
	
	\medskip
	
	The construction of $\otcA_g^{[m]}$ by {\sc Faltings} and {\sc Chai}~\cite[p.~128]{Faltings-Chai} yields a stack smooth over $\Spec \ZZ[\zeta_m,1/m]$, where $\zeta_m$ is a primitive $m$-th root of unity. Its boundary is a normal crossings divisor. However, their definition of full level-$m$ structure requires a symplectic isomorphism $A[m] \xrightarrow{\sim} (\ZZ/m\ZZ)^{2g}$. In~\cite[IV Remark~6.12]{Faltings-Chai}, they relax the requirement that the isomorphism be symplectic, giving a  stack smooth over $\Spec\ZZ[1/m]$; in~\cite[I Definition~1.8]{Faltings-Chai} they also consider full level structures in our sense (albeit still requiring the isomorphism~\eqref{eq:levelstructure} to be symplectic). Combining these remarks we obtain a stack  we denote $(\otcA_g^{[m]})_{\ZZ[1/m]}$, smooth over $\ZZ[1/m]$. If $m \geq 3$, this stack is a scheme~\cite[IV.6.9]{Faltings-Chai}.
	
	We extend the construction to $\Spec \ZZ$ by defining $\otcA_g^{[m]}$ to be the normalization of $\otcA_g$ in  $(\otcA_g^{[m]})_{\ZZ[1/m]}$. The resulting stack is not smooth over primes dividing $m$, and even the interior of the stack over such primes does not have a modular interpretation. However, the boundary structure of this stack at primes dividing $m$ is described in the Appendix.
		
	The natural morphism $\cA_g^{[m]} \to \cA_g$ that ``forgets the level structure" is finite, and since we chose compatible compactifications, it extends to a finite morphism $\pi_m\colon \otcA_g^{[m]} \to \otcA_g$. 

%%%%%%%%%%%%%%%%%%%%%%%%%%%%%%%%%%%%%%%%%%%%%%%%%%%%%%%%%%%%%%%%%%%%%%%%%

\subsection{Rational points and covers of bounded degree}
	\label{ss:extensions}

	The stack $\otcA_g$ is proper, but a rational point $x\in \otcA_g$ might not extend to an integral point: it might correspond to an abelian variety with {\em potentially} semistable, but not semistable, reduction. In this section we explain how one can use an integral extension of bounded degree to lift $x$ to a finite cover of $\otcA_g$ that is a scheme, where the lift of $x$ can be extended to an integral point. 

	We apply Lemma \ref{Lem:lifting-points}(3), which requires a  covering $Y \to \otcA_g$  by a scheme. This can be achieved using \cite[Theorem 1]{KreschVistoli}, but a more explicit construction in our situation is given in the following well-known lemma.

	\begin{lemma}
		\label{lem:products}
		Let $m = m_1m_2$ be a product of two coprime integers each $\geq 3$. Then the stack $\otcA_g^{[m]}$ is a scheme.\footnote{\ChDan{The authors would appreciate information on an early reference for this well-known argument.}}
	\end{lemma}

	\begin{proof}
		First, recall that if $d \geq 3$ is an integer, the stack $(\otcA_g^{[d]})_{\ZZ[1/d]}$ is a scheme. It suffices to show that $(\otcA_g^{[m]})_{\ZZ[1/m_1]}$ and $(\otcA_g^{[m]})_{\ZZ[1/m_2]}$ are schemes. This in turn follows because for $i = 1$ and $2$, the stack $(\otcA_g^{[m]})_{\ZZ[1/m_i]}$ is the normalization of the scheme $(\otcA_g^{[m_i]})_{\ZZ[1/m_i]}$ in the scheme $(\otcA_g^{[m]})_{\ZZ[1/m]}$.
	\end{proof}

	Since $12=3\cdot 4$ is the product of two relatively prime integers each $\geq 3$ it follows that $\otcA_g^{[12]}$ is a scheme. Let $M = \deg \pi_{12 }\colon \otcA_g^{[12]} \to \otcA_g$. We obtain: 

	\begin{proposition}
		\label{Prop:lift-points}	
		Let $R$ be a Dedekind domain with field of fractions $K$. Fix a point $y \in \otcA_g^{[m]}(K)$. There is a finite extension $L/K$, with $[L:K]\leq M$ and $R_L\subset L$ the integral closure of $R$, and a point $\eta \in \otcA_g^{[m]}(R_L)$ lifting $y$.
	\end{proposition}
	
	\begin{proof} 
		Applying Lemma \ref{Lem:lifting-points}(3) to the point $\pi_m(y) \in \otcA_g$, we have a point $\xi \in \otcA_g(R_L)$ lifting $\pi_m(y)$. Applying Lemma \ref{Lem:lifting-points}(1) to the representable morphism $\pi_m$, the point lifts to $\eta \in \otcA_g^{[m]}(R_L)$.
	\end{proof}
			
%%%%%%%%%%%%%%%%%%%%%%%%%%%%%%%%%%%%%%%%%%%%%%%%%%%%%%%%%%%%%%%%%%%%%%%%%

\subsection{Substacks}
	\label{ss:notation}

	Let $X \subseteq (\tcA_g)_K$ be a closed substack, let $X'\to X$ be a resolution of singularities, $X' \subset \overline X'$ a smooth compactification with $D = \overline X' \smallsetminus X'$ a normal crossings divisor. Assume that the rational map $f\colon\oX' \to \otcA_g$ is a morphism. Let $X'_m = X'\times_{\tcA_g}\tcA_g^{[m]}$, and let $\oX'_m \to\oX'\times_{\otcA_g} \ocA_g^{[m]}$ be a resolution of singularities with projections $\pi_m^X\colon \oX'_m\to \overline{X}'$ and $f_m\colon \oX'_m\to \otcA_g^{[m]}$. 	

	We now spread these objects over $\cO_{K,S}$ for a suitable finite set of places $S$ containing the archimedean places. Let $(\sX,\sD)$ be a normal crossings model of $(\oX',D)$ over $\Spec \cO_{K,S}$. As above, write $\sD = \sum_i \sD_i$. Such a model exists, even for Deligne-Mumford stacks, by \cite[Proposition~2.2]{Olsson}.

	Let $X(K)_{[m]}$ be the set of $K$-rational points of $X$ corresponding to abelian varieties $A/K$ admitting full level-$m$ structure. Define
	\begin{equation}
		\label{eq:toinfty}
		X(K)_{p\geq m_0} := \bigcup_{\substack{p \geq m_0 \\ p \textrm{ prime}}} X(K)_{[p]}.
	\end{equation}

%%%%%%%%%%%%%%%%%%%%%%%%%%%%%%%%%%%%%%%%%%%%%%%%%%%%%%%%%%%%%%%%%%%%%%%%%

\subsection{Intersection Multiplicities for integral and rational points}
	\label{ss:mult}
			
	Write $E$ for the boundary divisors of $\left(\otcA_g\right)_K$, and $\sE$ for its closure in $\otcA_g$, which is a Cartier divisor. We have an equality of divisors on $\oX'$
	\[
		f^*E = \sum a_iD_i,
	\]
where each $a_i>0$; see \cite[Equation (4.3)]{Alevels}. This equality extends over $\Spec \cO_{K,S}$ to
	\[
		f^*\sE = \sum a_i\sD_i.
	\]
By \cite[Proposition~4.1 or Equation~(4.1)]{Alevels}, we have that $\pi_m^* E = m E_m$ for some Cartier divisor $E_m\subset \left(\otcA_g^{[m]}\right)_K$. Spreading out $E_m$ to $\sE_m$ in $\otcA_g^{[m]}$ we obtain $\pi^*_m\sE = m\sE_m$; moreover, by Proposition \ref{prop:full level structure} in the appendix, $\sE_m$ is a Cartier divisor.

	Let $q \subset \cO_{K,S}$ be a nonzero prime ideal. Assume there are maps $\xi\colon\Spec \cO_{K,q} \to \sX$ and $\xi_m\colon\Spec \cO_{K,q} \to \sX_m$ such that $\xi = \pi_m^X\circ\xi_m$, and write $x \in \sX(\cO_{K,q})$ and $x_m \in \sX_m(\cO_{K,q})$ for the respective \emph{integral} points corresponding to $\xi$ and $\xi_m$. These objects and arrows fit together in the commutative diagram
	\[
		\xymatrix{
			& \sX_m\ar[r]^{f_m} \ar[d]^{\pi_m^X} & \otcA_g^{[m]} \ar[d]^{\pi_m} \\
			\Spec \cO_{K,q} \ar[ru]^{\xi_m}\ar[r]_\xi & \sX \ar[r]^f & \otcA_g
		}
	\]
We have an equality of divisors on $\Spec \cO_{K,q}$:
	\[
		\xi^* f^* \sE \,=\, \xi_m^* f_m^* \pi_m^* \sE \,=\, m \cdot\xi_m^* f_m^* \sE_m,
	\]
which translates to 
	\begin{equation*}
		\label{eq:multsatpts}
		\sum a_i \xi^* \sD_i = m \cdot \xi_m^* f_m^* \sE_m.
	\end{equation*}
The divisor on the left has multiplicity $\sum a_i n_q(\sD_i,x)$. If $x \in \sX(\cO_{K,q})$, then the intersection multiplicities $n_q(\sD_i,x)$ are integers, and we deduce that 
	\[
		m\ \Big|\ \sum a_in_q(\sD_i,x).
	\]
If the quantity $\sum a_i n_q(\sD_i,x)$ is nonzero then $m \leq  \sum a_i n_q(\sD_i,x)$, and thus
	\[
		m \leq \max\{a_i\} \sum n_q(\sD_i,x) = \max\{a_i\} n_q(\sD,x),
	\]
in other words,
	\[
		n_q(\sD,x) \geq \frac{m}{\max\{a_i\}}.
	\]

	Given a \emph{rational} point $x \in \sX(K)$, we apply Proposition \ref{Prop:lift-points}	and obtain an extension field $L/K$ with $[L:K] \leq M$ and an integral extension $\cO_{L,q}$ with a point $\xi \in \sX(\cO_{L,q})$ lifting $x$. Since for any $\frq \mid  q$ we have $e(\frq\mid q) \leq M$, Equation \eqref{Eq:n_q-extension} gives 
	\[
		n_q(\sD,x) \geq \frac{m}{M\max\{a_i\}}.
	\]
We summarize this discussion in the following proposition.

	\begin{proposition}
		\label{prop:betterthanzero}
		With notation as in \S\ref{ss:notation},  write $\alpha(X): = (M\cdot \max\{a_i\})^{-1} > 0$, which depends $X$, but not on $x$.  Let $x_m \in X_m'(K)$ be a rational point in $X'_m$ with image $x \in X'(K)$. Suppose that $n_q(\sD,x) > 0$. Then
		\begin{equation}
			\label{eq:boundingintersecmults}
			n_q(\sD,x) \geq  m\alpha(X).
		\end{equation}
	\end{proposition}

%%%%%%%%%%%%%%%%%%%%%%%%%%%%%%%%%%%%%%%%%%%%%%%%%%%%%%%%%%%%%%%%%%%%%%%%%

\subsection{Proof of Theorem~\ref{thm:main}}

	\begin{lemma}
		\label{lem:primelevels}
		Fix $\epsilon' > 0$. Then there is an integer $m_0 := m_0(\epsilon',K,X)$, such that for all primes $p \geq m_0$ and $x \in X(K)_{[p]}$ we have
		\[
			 d_K(\cT_x) \leq h_{\epsilon' D}(x) + O(1).	
		\]
	\end{lemma}

	\begin{proof}
		Let $A/K$ be the abelian variety of dimension $g$ associated with $x \in X(K)_{[p]}$. Since $A$ has full level-$p$ structure, we know that $\#A[p](K) \geq p^g$. Thus, if $\frq$ is a prime ideal of $K$ that does not divide $p$, then $\#A[p](\kappa(\frq)) \geq p^g$ (see~\cite[C.1.4]{HindrySilverman}). We choose $m_0 \geq 8$, so $p\neq 2$, freeing us to pick $\frq \mid 2$. This implies that $\kappa(\frq) = 2^{f(\frq|q)} \leq 2^{[K:\QQ]}$.
	
		We follow {\sc Flexor--Oesterl\'e} \cite[Th\'eor\`eme  3]{Flexor-Oesterle} and {\sc Silverberg} \cite[Theorem 3.3]{Silverberg-finite-order}, see also  {\sc Kamienny} \cite[\S6(2a)]{Kamienny-root-p}. Suppose now that $A$ has good reduction at $\frq$, so that, by the {\sc Lang-Weil} estimates, we have 
		\[
			\#A(\kappa(\frq)) \leq (1 + \kappa(\frq)^{1/2})^{2g} \leq (1 + 2^{[K:\QQ]/2})^{2g}.
		\] 
Thus, if $A$ has good reduction at $\frq\mid 2$, we have
		\[
			p \leq (1 + 2^{[K:\QQ]/2})^2 := \gamma.
		\]
In other words, if $p > \gamma$, then $A$ must have bad reduction at primes $\frq \mid 2$, so $n_\frq(\sD,x) > 0$. By Proposition~\ref{prop:betterthanzero} the stronger inequality~\eqref{eq:boundingintersecmults} holds with $m = p$. We use this to see that if $p > \gamma$, then as in the estimate \eqref{Eq:counting-height} we have
		\begin{equation}
			\label{eq:epsm}
			h_{\epsilon' D}(x) + O(1)\geq 
			\epsilon'\sum_{\frq} n_\frq(D,x) \log |\kappa(\frq)| \geq \ 
			\epsilon'\sum_{\frq \mid 2} p\alpha(X) \log |\kappa(\frq)| \geq 
			\epsilon'(\alpha(X)\log 2) \cdot p,
		\end{equation}
so  $h_{\epsilon' D}(x)$ grows at least linearly in $p$.
	
		Now we crudely bound $d_K(\cT_x)$ from above. Note that $x$ is an integral point away from $p$. As in~\S\ref{ss:mult}, passing to a cover of finite \emph{bounded} degree $\leq M = M(g)$, we may replace $x$ with an integral point $y$ in such a way that $[K(y):K] \leq M$. The discriminant ideal of $\cT_x$ divides the discriminant ideal of the extension $K(y)/K$; we compare their factors at $p$. Let $d_K(K(y))_{p}$ denote the contribution at $p$ of $d_K(K(y))$; ignoring negative terms coming from the discriminant of $\cO_K$, we have the estimate 
		\[
			d_K(\cT_x) \leq d_K(K(y))_{p} \leq \frac{v_p(|\Disc(\cO_{K(y)})|)}{[K(y):K]}\cdot \log p,
		\]
where $v_p$ denotes the usual $p$-adic valuation. By~\cite[Proof of~III.2.13]{Neukirch}, we have 
		\[
			v_p(|\Disc(\cO_{K(y)})|) \leq [K(y):K](1 + [K(y):K])
		\]
Hence
		\begin{equation}
			\label{eq:discbounded}
			d_K(\cT_x) \leq (1 + [K(y):K])\cdot \log p \ := \beta\cdot \log p
		\end{equation}
grows at most linearly in $\log p$, and the result follows. 
	\end{proof}

	\begin{lemma}
		\label{lem:primelevels-truncated}
		Fix $\epsilon' > 0$. Then there is an integer $m_0 := m_0(\epsilon',K,X)$, such that for all primes $p \geq m_0$, if $x \in X(K)_{[p]}$ then
		\[
			N^{(1)}_K(D,x)\ \  \leq\ \  h_{\epsilon' D}(x) + O(1).
		\]
	\end{lemma}

	\begin{proof}
		If $x \in X(K)_{[p]}$, then whenever  $n_q(\sD,x) > 0$,  Proposition~\ref{prop:betterthanzero} implies that the stronger inequality~\eqref{eq:boundingintersecmults} holds. Hence
		\begin{align*}
			p\alpha(X)\,N^{(1)}_K(D,x) &= \sum_{n_q(\sD,x) > 0} p\alpha(X) \log |\kappa(q)|\\
			&\leq\sum_{n_q(\sD,x) > 0} n_q(\sD,x) \log |\kappa(q)|\\
			&\leq h_{D}(x) + O(1),
		\end{align*}
where in the last inequality  we use the estimate \eqref{Eq:counting-height}. Taking $m_0 > 1/(\epsilon'\alpha(X))$ we have $p\alpha(X) > 1/\epsilon'$, hence $N^{(1)}_K(D,x) \leq h_{\epsilon'D}(x) + O(1)$.
	\end{proof}

	\begin{proof}[Proof of Theorem~\ref{thm:main}]
		We proceed by Noetherian induction. For each integer $i \geq 1$, let
		\[
			W_i = \overline{\tcA_g(K)_{p\geq i}}.
		\]
Note that $W_i$ is a closed subset of $\cA_g$, and that $W_i \supseteq W_{i+1}$ for every $i$. The chain of $W_i$ must stabilize by the Noetherian property of the Zariski topology of $\cA_g$.  Say $W_n = W_{n+1} = \cdots$. 

		We claim that $W_n$ has dimension $\leq 0$. Suppose not, and let $X \subseteq W_n$ be an irreducible component of positive dimension. Fix $\epsilon > 0$ so that $K_X + (1 - \epsilon)D$ is big: such an $\epsilon$ exists by~\cite[Corollary~1.7]{Alevels}. Next, choose a $\QQ$-ample divisor $H$ such that $K_X + (1-\epsilon)D - H$ is effective, and apply Proposition~\ref{prop:VojtaToVojtaDM}, with $r = 1$, to conclude there is a Zariski-closed proper subset $Z \subset X$ such that if $x \in X(K) \smallsetminus Z(K)$, then
		\[
			N_K^{(1)}(D,x) + d_K(\cT_x)\ \  \geq\ \  h_{K_X(D)}(x) - \delta h_H(x) - O(1).
		\]
By Lemma \ref{lem:primelevels-truncated}, for all primes $p > m_0$ any $x \in X(K)_{[p]}$
 satisfies $N^{(1)}_K(D,x) \leq h_{(\epsilon/2)D}(x) + O(1)$. On the other hand, Lemma~\ref{lem:primelevels} guarantees that, after possibly enlarging  $m_0$, for all primes $p \geq m_0$ any $x \in X(K)_{[p]}$ satisfies $h_{(\epsilon/2) D}(x) + O(1) \geq d_K(\cT_x)$.  If also $x \notin Z(K)$ we deduce that
		\[
			  h_{\epsilon D}(x) \ \ \geq\ \  h_{K_X(D)}(x) - \delta h_H(x) - O(1).
		\]
By our choice of $H$ and~\cite[Theorem B.3.2(e)]{HindrySilverman}, we obtain
		\[
			O(1) \geq (1 - \delta)h_{K_X((1 -\epsilon )D)}(x).
		\]
Using~\cite[Theorem~B.3.2(e,g)]{HindrySilverman} we conclude that the set of $x \in X(K)_{p\geq m_0}$ outside $Z(K)$ is not dense, and thus $X(K)_{p\geq m_0}$ is contained in a Zariski-closed proper subset of $X$. On the other hand, if $m_0 > n$, then $W_{m_0} = W_n$, so $X$ is also an irreducible component of $W_{m_0}$, and hence $\overline{X(K)_{p\geq m_0}} = X$, a contradiction. This proves that $\dim W_n \leq 0$.
		
		Finally, if $W_n$ is a finite set of points, then it is well-known  that the full level structures that can possibly appear in any of the corresponding finitely many geometric isomorphism classes are bounded. Indeed if $\frq\in M_K^0$ is a fixed prime of potential good reduction, all twists with full level-$p$  structure with $p>2, \frq\nmid p$ have good reduction at $\frq$. Since the $p$-torsion points inject modulo $\frq$ we have $p\leq (1+ N\frq^{1/2})^2.$ Alternatively, following {\sc Manin} \cite[\S 3]{Manin}, there are only finitely many isomorphism classes over $K_\frq$, and for each the torsion subgroup is finite.
	\end{proof}

\appendix

\section{Compactifications with full level structure \\ by Keerthi Madapusi Pera}

The purpose of this appendix is to lay out certain facts about toroidal compactifications of the moduli of principally polarized abelian varieties with full level structure at `bad' primes. This is a straight-forward extension of the theory of~\cite{Faltings-Chai}, and could possibly also be extracted from the work of K.-W. Lan~\cite{Lan2016-bd}. 

\subsection{}
Equip $\ZZ^{2g} = \ZZ^g\oplus\ZZ^g$ with the standard non-degenerate symplectic pairing
\[
\psi\colon ((u_1,v_1),(u_2,v_2))\mapsto u_1v_2^t - u_2v_1^t.
\]
For every integer $m\in\ZZ_{>0}$,  equip $(\ZZ/m\ZZ)^{2g}$ with the non-degenerate pairing $\psi_m$ inherited from $\psi$. A \textsf{symplectic level-$m$} structure on a principally polarized abelian scheme $(A,\lambda)$ over a base $S$ will consist of a pair $(\eta,\phi)$, where 
\[
\eta\colon\underline{(\ZZ/m\ZZ)}^{2g}\xrightarrow{\simeq}A[m]\;;\;\phi\colon\mu_m\xrightarrow{\simeq}\underline{\ZZ/m\ZZ}
\]
are isomorphisms of group schemes over $S$ such that $\eta$ carries the pairing $\psi_m$ to the pairing $\phi\circ e_\lambda$ on $A[m]$. Here,
\[
e_\lambda\colon A[m]\times A[m]\to \mu_m
\]
is the symplectic Weil pairing induced by the polarization $\lambda$.

\subsection{}
Let $\widetilde{\mathcal{A}}_g$ be the algebraic stack over $\ZZ$ parameterizing principally polarized abelian varieties of dimension $g$. Over $\ZZ[1/m]$, we have a finite \'etale morphism of algebraic stacks 
\[
\widetilde{\mathcal{A}}_{g,m}[1/m]\to \widetilde{\mathcal{A}}_g[1/m] 
\]
parameterizing symplectic level-$m$ structures on the universal abelian scheme over $\widetilde{\mathcal{A}}_g[1/m]$. By a classical argument of Serre, points of $\widetilde{\mathcal{A}}_{g,m}[1/m]$ have trivial automorphism schemes as soon as $m\geq 3$. 

Fix any toroidal compactification $\overline{\widetilde{\mathcal{A}}}_g$ of $\widetilde{\mathcal{A}}_g$ (see~\cite[Ch. IV]{Faltings-Chai}). We now obtain an open immersion
\[
\widetilde{\mathcal{A}}_{g,m}\hookrightarrow\overline{\widetilde{\mathcal{A}}}_{g,m}
\]
of algebraic stacks over $\ZZ$ by taking the normalization of the open immersion
\[
\widetilde{\mathcal{A}}_{g}\hookrightarrow\overline{\widetilde{\mathcal{A}}}_{g}
\]
in $\widetilde{\mathcal{A}}_{g,m}[1/m]$.

The stack $\widetilde{\mathcal{A}}_{g,m}$ has no obvious moduli interpretation over $\ZZ$, and we know little about the singularities of its fibers over primes dividing $m$. However, this is not an obstruction to studying its general structure at the boundary. For this, we will need some information about the stratification of the boundary. 

\subsection{}
We direct the reader to~\cite[\S 1]{mp:toroidal} for the notion of a \textsf{principally polarized $1$-motif $(Q,\lambda)$} over a base $S$. Here, we will note that it consists of a $1$-motif $Q$---that is, a two-term complex $u\colon X\to J$, where $J$ is a semi-abelian scheme over $S$ that is an extension of an abelian scheme by a torus, and $X$ is a locally constant sheaf of finite free abelian groups---and an isomorphism $\lambda\colon Q\xrightarrow{\simeq}Q^\vee$ to its dual $1$-motif $Q^\vee$. 

We will say that $(Q,\lambda)$ is \textsf{of type $(r,s)$} for $r,s\in\ZZ_{\geq 0}$ if the abelian part of $J$ has dimension $s$ and if $X = \underline{\ZZ}^r$. The polarization $\lambda$ then canonically identifies the toric part of $J$ with $\Gm^r$. 

Suppose that $(Q,\lambda)$ is of type $(r,s)$, and set $g=r+s$. Given $m\in\ZZ_{>0}$, one has the $m$-torsion $Q[m]$ of the $1$-motif $Q$: this is a finite flat group scheme over $S$ of rank $2g$, and the polarization equips it with a non-degenerate Weil pairing $e_\lambda$ with values in $\mu_m$. 

Let $B$ be the abelian part of $J$. Then there is a natural ascending $3$-step filtration 
\[
0= W_{-3}Q[m] \subset W_{-2}Q[m] = \mu_m^r\subset W_{-1}Q[m]\subset W_0Q[m] = Q[m],
\]
where $W_{-2}Q[m]$ is isotropic for the Weil pairing, $W_{-1}Q[m]$ is its orthogonal complement, $\mathrm{gr}^W_{-1}Q[m]$ is identified with $B[m]$, compatibly with Weil pairings, and $\mathrm{gr}^W_0Q[m]$ is identified with $(\ZZ/m\ZZ)^r$. The induced pairing
\[
 (\ZZ/m\ZZ)^r\times \mu_m^r = \mathrm{gr}^W_0Q[m] \times W_{-2}Q[m]\xrightarrow{e_\lambda}\mu_m
\]
is the canonical one. 

Let $I_r\subset\ZZ^{2g}$ be the isotropic subspace spanned by the first $r$ basis vectors of the first copy of $\ZZ^g$. We have identifications
\[
\ZZ^r = I_r\;;\; \ZZ^r = \ZZ^{2g}/I_r^{\perp},
\]
so that the induced non-degenerate pairing
\[
\ZZ^r \times \ZZ^r = I_r \times \ZZ^{2g}/I_r^{\perp} \xrightarrow{\psi} \ZZ.
\]
is the standard symmetric pairing $(u,v)\mapsto uv^t$.

A \textsf{symplectic level-$m$} structure on $(Q,\lambda)$ is a pair $(\eta,\phi)$, where
\[
\eta\colon\underline{(\ZZ/m\ZZ)}^{2g}\xrightarrow{\simeq}Q[m]\;;\;\phi\colon\mu_m\xrightarrow{\simeq}\underline{\ZZ/m\ZZ}
\]
are isomorphisms of group schemes over $S$ such that $\eta$ carries the pairing $\psi_m$ to the pairing $\phi\circ e_\lambda$ on $Q[m]$ and the subspace $I_r/mI_r$ onto $W_{-2}Q[m]$, so that the induced isomorphism
\[
(\ZZ/m\ZZ)^r = I_r/mI_r \xrightarrow{\simeq} W_{-2}Q[m] = \mu_m^r \xrightarrow[\phi^{-1}]{\simeq}(\ZZ/m\ZZ)^r
\]
is the identity.

We now obtain a moduli stack $\widetilde{\mathcal{Y}}_{r,s}$ over $\ZZ$ of principally polarized $1$-motifs, and a finite \'etale cover
\[
\widetilde{\mathcal{Y}}_{r,s,m}[1/m]\to \widetilde{\mathcal{Y}}_{r,s}[1/m] 
\]
over $\ZZ[1/m]$, parameterizing symplectic level-$m$ structures on the universal principally polarized $1$-motif.

\subsection{}
Consider the moduli stack $\widetilde{\mathcal{Y}}_{r,0}$: This parameterizes principally polarized $1$-motives of the form $u\colon\ZZ^r\to \Gm^r$. Alternatively, it parameterizes symmetric pairings $\ZZ^r\times \ZZ^r\to\Gm$. As such, it is represented over $\ZZ$ by the torus with character group 
$\mathsf{S}_r = \mathrm{Sym}^2\ZZ^r$. 

Similarly, by the discussion in \cite[Ch. IV, \S 6.5]{Faltings-Chai}, the morphism 
\[
\widetilde{\mathcal{Y}}_{r,0,m}[1/m]\to \widetilde{\mathcal{Y}}_{r,0}[1/m] 
\]
parameterizes lifts $\frac{1}{m}\ZZ^r\to \Gm^r$ of the universal homomorphism $\ZZ^r\to\Gm^r$, and so is represented over $\ZZ[1/m]$ by the torus with character group $(1/m)\mathsf{S}_r$. The natural map
\[
\widetilde{\mathcal{Y}}_{r,0,m}[1/m]\to \widetilde{\mathcal{Y}}_{r,0}[1/m] 
\]
corresponds to the map of tori induced by the inclusion $\mathsf{S}_r\hookrightarrow (1/m)\mathsf{S}_r$ of character groups. 

Therefore, the normalization $\widetilde{\mathcal{Y}}_{r,0,m}$ of $\widetilde{\mathcal{Y}}_{r,0}$ in $\widetilde{\mathcal{Y}}_{r,0,m}[1/m]$ is represented over $\ZZ$ by the torus with character group $(1/m)\mathsf{S}_r$, and is in particular smooth over $\ZZ$.

\subsection{}
When $s>0$, $\widetilde{\mathcal{Y}}_{r,s}$ permits a similar, but slightly more elaborate description. We have the obvious map $\widetilde{\mathcal{Y}}_{r,s}\to \widetilde{\mathcal{A}}_s$ assigning to a polarized $1$-motif $(Q,\lambda)$ of type $(r,s)$ the abelian part of the semi-abelian scheme $J$. 

There is a natural action of the torus $\mathcal{Y}_{r,0}$ on $\mathcal{Y}_{r,s}$: Given a polarized $1$-motif $(Q_0,\lambda_0)$ of type $(r,0)$ associated with a homomorphism $u_0\colon\underline{\ZZ}^r\to \Gm^r$ and a polarized $1$-motif $(Q,\lambda)$ of type $(r,s)$ associated with $u\colon\underline{\ZZ}^r\to J$, the product $u_0\cdot u\colon\underline{\ZZ}^r\to J$ corresponds to another principally polarized $1$-motif of type $(r,s)$.

The quotient of $\widetilde{\mathcal{Y}}_{r,s}$ by this action is naturally identified with the abelian scheme $\widetilde{\mathcal{C}}_{r,s}\to \widetilde{\mathcal{A}}_s$ that parameterizes homomorphisms 
\[
v\colon\ZZ^r\to B, 
\]
where $B$ is the universal abelian scheme over $\widetilde{\mathcal{A}}_s$. So we obtain a tower of algebraic stacks:
\begin{equation}\label{eqn:tower no level}
\widetilde{\mathcal{Y}}_{r,s}\to \widetilde{\mathcal{C}}_{r,s}\to \widetilde{\mathcal{A}}_s,
\end{equation}
where the first morphism is a $\widetilde{\mathcal{Y}}_{r,0}$-torsor, and the second is an abelian scheme. 

From the discussion in~\cite[Ch. IV,\S 6.5]{Faltings-Chai}, we find that the stack $\widetilde{\mathcal{Y}}_{r,s,m}[1/m]$ admits a compatible tower structure:
\begin{equation}\label{eqn:tower level m}
\widetilde{\mathcal{Y}}_{r,s,m}[1/m]\to \widetilde{\mathcal{C}}_{r,s,m}[1/m]\to \widetilde{\mathcal{A}}_{s,m}[1/m].
\end{equation}
Here, $\widetilde{\mathcal{C}}_{r,s,m}[1/m]$ parameterizes homomorphisms
\[
v_m\colon \frac{1}{m}\ZZ^r\to B,
\]
where $B$ is the universal abelian scheme over $\widetilde{\mathcal{A}}_{s,m}[1/m]$, and $\widetilde{\mathcal{Y}}_{r,s,m}[1/m]$ parameterizes homomorphisms
\[
u_m\colon\frac{1}{m}\ZZ^r\to J,
\]
lifting $v_m$, where $J$ is the universal semi-abelian scheme over $\widetilde{\mathcal{C}}_{r,s,m}[1/m]$ parameterized by the homomorphism
\[
m\cdot v_m\colon \ZZ^r\to B \xrightarrow{\simeq} B^\vee.
\]
It is therefore naturally a $\widetilde{\mathcal{Y}}_{r,0,m}[1/m]$-torsor over $\widetilde{\mathcal{C}}_{r,s,m}[1/m]$.

From this description, it is clear that the normalization of the tower~\eqref{eqn:tower no level} in the tower~\eqref{eqn:tower level m} gives us a tower 
\begin{equation}\label{eqn:tower level m integral}
\widetilde{\mathcal{Y}}_{r,s,m}\to \widetilde{\mathcal{C}}_{r,s,m}\to \widetilde{\mathcal{A}}_{s,m},
\end{equation}
where 
\[
\widetilde{\mathcal{C}}_{r,s,m}\to \widetilde{\mathcal{A}}_{s,m} 
\]
is still an abelian scheme parameterizing homomorphisms $v_m\colon\frac{1}{m}\ZZ^r\to B$ (with $B$ the universal abelian scheme over $\widetilde{\mathcal{A}}_{s,m}$), and 
\[
\widetilde{\mathcal{Y}}_{r,s,m}\to \widetilde{\mathcal{C}}_{r,s,m}
\]
is once again a $\widetilde{\mathcal{Y}}_{r,0,m}$-torsor parameterizing lifts $u_m\colon\frac{1}{m}\ZZ^r\to J$ of $v_m$, where $J$ is still classified by $v = m\cdot v_m$.

In particular, the morphism 
\begin{equation}\label{eqn:pushforward}
\widetilde{\mathcal{Y}}_{r,s,m}\to \widetilde{\mathcal{Y}}_{r,s}\times_{\widetilde{\mathcal{C}}_{r,s}}\widetilde{\mathcal{C}}_{r,s,m}
\end{equation}
is obtained via pushforward of torsors along the morphism
\[
\widetilde{\mathcal{Y}}_{r,0,m}\to \widetilde{\mathcal{Y}}_{r,0}
\]
of tori, which is of course canonically isomorphic to the multiplication-by-$m$ map
\[
\widetilde{\mathcal{Y}}_{r,0}\xrightarrow{[m]}\widetilde{\mathcal{Y}}_{r,0}.
\]

\subsection{}
Fix a rational polyhedral cone $\sigma\subset (\mathsf{S}_r)_{\QQ}Q$: this gives us twisted toric embeddings
\[
\widetilde{\mathcal{Y}}_{r,s}\hookrightarrow \widetilde{\mathcal{Y}}_{r,s}(\sigma)\;;\; \widetilde{\mathcal{Y}}_{r,s,m}\hookrightarrow \widetilde{\mathcal{Y}}_{r,s,m}(\sigma).
\]
The complements of these embeddings admit a natural stratification with a unique closed stratum, which we denote by $\mathcal{Z}_{r,s}(\sigma)$ and $\mathcal{Z}_{r,s,m}(\sigma)$, respectively.

Let $\widehat{\widetilde{\mathcal{Y}}}_{r,s}(\sigma)$ and $\widehat{\widetilde{\mathcal{Y}}}_{r,s}(\sigma)$ be the formal completions of $\widetilde{\mathcal{Y}}_{r,s}(\sigma)$ and $\widetilde{\mathcal{Y}}_{r,s,m}(\sigma)$, respectively, along their closed strata. By abuse of notation, write$\widehat{\widetilde{\mathcal{Y}}}_{r,s,m}(\sigma)[1/m]$  for the completion of  $\widetilde{\mathcal{Y}}_{r,s,m}(\sigma)[1/m]$ along its closed stratum.

Note that the morphism
\begin{equation}\label{eqn:pushforward sigma}
\widetilde{\mathcal{Y}}_{r,s,m}(\sigma)\to \widetilde{\mathcal{Y}}_{r,s}(\sigma)\times_{\widetilde{\mathcal{C}}_{r,s}}\widetilde{\mathcal{C}}_{r,s,m}
\end{equation}
is obtained via contraction along the multiplication-by-$m$ map on $\widetilde{\mathcal{Y}}_{r,0}$.

\subsection{}
Let $\Gamma(\sigma)\subset\mathrm{GL}_r(\ZZ)$ be the stabilizer of $\sigma$, and let $\Gamma_m(\sigma)\leq\Gamma(\sigma)$ be the subgroup of matrices that are trivial mod $m$: these are both finite groups, and $\Gamma_m(\sigma)$ is trivial as soon as $m\geq 3$.

By the main results of~\cite[Ch. IV]{Faltings-Chai}, the toroidal compactification $\overline{\widetilde{\mathcal{A}}}_{g}$ admits a stratification by locally closed substacks $\mathcal{Z}(r,\sigma)$ equipped with an isomorphism to $\Gamma(\sigma)\backslash\mathcal{Z}_{r,g-r}(\sigma)$ for some $r\leq g$, and some $\sigma\subset (\mathsf{S}_r)_{\QQ}$, and such that this isomorphism extends to one of formal completions
\[
\bigl(\overline{\widetilde{\mathcal{A}}}_{g}\bigr)^\wedge_{\mathcal{Z}(r,\sigma)}\xrightarrow{\simeq}\Gamma(\sigma)\backslash\widehat{\widetilde{\mathcal{Y}}}_{r,g-r}(\sigma).
\]
Faltings and Chai use the language of degeneration data. For a formulation using our language of $1$-motifs, we guide the reader to \cite[\S 3.1.5]{stroh:thesis}. 

The main idea is that, on every formally \'etale affine chart
\[
\mathrm{Spf}(R,I) \to \widehat{\widetilde{\mathcal{Y}}}_{r,g-r}(\sigma)
\]
one obtains a principally polarized $1$-motif $(Q,\lambda)$ of type $(r,g-r)$ over the fraction field $K(R)$ of $R$ associated with a semi-abelian scheme $J\to \Spec R$, and a period map $u\colon\ZZ^r\to J(K(R))$. This period map `degenerates' along $\Spec R/I$, and a construction of Mumford, explained in~\cite[Ch. III]{Faltings-Chai}, now gives us a principally polarized abelian scheme $(A,\psi)$ over $K(R)$ with semi-abelian degeneration over $R$, and equipped with a canonical symplectic identification $Q[m]\xrightarrow{\simeq}A[m]$, for every integer $m$. The pair $(A,\psi)$ now gives a map $\Spec K(R)\to \widetilde{\mathcal{A}}_g$, which extends to a map 
\[
\Spec R\to \overline{\widetilde{\mathcal{A}}}_{g},
\]
which in turn induces a map
\[
\mathrm{Spf}(R,I)\to \bigl(\overline{\widetilde{\mathcal{A}}}_{g}\bigr)^\wedge_{\mathcal{Z}(r,\sigma)}
\]
of formal algebraic stacks. These maps are now glued together to give the inverse of the desired isomorphism of formal neighborhoods.

Similarly, $\overline{\widetilde{\mathcal{A}}}_{g,m}[1/m]$ admits a compatible stratification by locally closed substacks $\mathcal{Z}_m(r,\sigma)[1/m]$ equipped with an isomorphism to $\Gamma(\sigma)\backslash\mathcal{Z}_{r,g-r,m}(\sigma)[1/m]$, and such that this isomorphism extends to one of formal completions
\[
\bigl(\overline{\widetilde{\mathcal{A}}}_{g,m}[1/m]\bigr)^\wedge_{\mathcal{Z}_m(r,\sigma)[1/m]}\xrightarrow{\simeq}\Gamma_m(\sigma)\backslash\widehat{\widetilde{\mathcal{Y}}}_{r,g-r,m}(\sigma)[1/m].
\]

\begin{proposition}
\label{prop:strat extend}
The stratification on $\overline{\widetilde{\mathcal{A}}}_{g,m}[1/m]$ extends to one of $\overline{\widetilde{\mathcal{A}}}_{g,m}$ by substacks $\mathcal{Z}_m(r,\sigma)$ equipped with an isomorphism to $\mathcal{Z}_{r,g-r,m}(\sigma)$, extending to an isomorphism
\[
\bigl(\overline{\widetilde{\mathcal{A}}}_{g,m}\bigr)^\wedge_{\mathcal{Z}_m(r,\sigma)}\xrightarrow{\simeq}\Gamma_m(\sigma)\backslash\widehat{\widetilde{\mathcal{Y}}}_{r,g-r,m}(\sigma).
\]
\end{proposition}
\begin{proof}
Let 
\[
\mathrm{Spf}(R,I) \to \widehat{\widetilde{\mathcal{Y}}}_{r,g-r,m}(\sigma)
\]
be a formally \'etale affine chart. The tautological principally polarized $1$-motif $(Q,\lambda)$ over $\Spec K(R)$ is now equipped with a canonical symplectic level $m$ structure, which in turn  also equips the principally polarized abelian scheme $(A,\psi)$, obtained from it via Mumford's construction, with a symplectic level $m$ structure. 

This implies that the associated map $\Spec K(R)\to \widetilde{\mathcal{A}}_{g}$ has a canonical lift
\[
\Spec K(R)\to \widetilde{\mathcal{A}}_{g,m},
\]
which then extends to a map $\Spec R\to \overline{\widetilde{\mathcal{A}}}_{g,m}$. 

Assume now that $R$ is a complete local ring of $\widetilde{\mathcal{Y}}_{r,g-r,m}(\sigma)$ with maximal ideal $I$ and algebraically closed residue field, and let $R'$ be the complete local ring of $\overline{\widetilde{\mathcal{A}}}_{g,m}$ at the image of the geometric closed point of $\Spec R$. We claim that the induced map $R'\to R$ is an isomorphism. This follows from two observations: First, it is a \emph{finite} map of normal local rings. Second, by the description of the stratification in characteristic $0$, if $p$ is the residue characteristic of $R$, then for any maximal ideal $\mathfrak{m}'\subset R'[1/p]$, the ideal $\mathfrak{m} = \mathfrak{m}'R[1/p]$ is once again maximal, and the induced map
\[
\widehat{R'[1/p]}_{\mathfrak{m}'} \to \widehat{R[1/p]}_{\mathfrak{m}}
\]
is an isomorphism. The second assertion shows, via faithfully flat descent, that every element of $R$ is contained in $R'[1/p]$, and the first shows that it must already be contained in $R'$.

Let $\eta_m\colon\overline{\widetilde{\mathcal{A}}}_{g,m}\to \overline{\widetilde{\mathcal{A}}}_{g}$ be the natural finite map. Combining the previous paragraph with Artin approximation, we find that $\eta_m$ is \'etale locally isomorphic to the finite map
\begin{align}\label{eqn:Ym to Y}
\mathcal{Y}_{r,g-r,m}(\sigma)\to \mathcal{Y}_{r,g-r}(\sigma)
\end{align}
for varying choices of $r$ and $\sigma$.

We claim that the reduced stack $\mathcal{Z}_m(r,\sigma)$ underlying the locally closed substack $\eta_m^{-1}(\mathcal{Z}(r,\sigma))\subset \overline{\widetilde{\mathcal{A}}}_{g,m}$ is normal. This can be checked on complete local rings using the observation that the reduced substack underlying the pre-image of $\mathcal{Z}_{r,g-r}(\sigma)$ under the map~\eqref{eqn:Ym to Y} is normal.

Moreover, from this and the fact that the locally closed substacks $\mathcal{Z}(r,\sigma)$ stratify $\overline{\widetilde{\mathcal{A}}}_g$, one can deduce that the locally closed substacks $\mathcal{Z}_m(r,\sigma)$ stratify $\overline{\widetilde{\mathcal{A}}}_{g,m}$.

By normality of the target, the map
\[
\mathcal{Z}_{r,g-r,m}(\sigma)[1/m]\xrightarrow{\simeq}\mathcal{Z}_m(r,\sigma)[1/m]\hookrightarrow \overline{\widetilde{\mathcal{A}}}_{g,m}
\]
extends uniquely to a map
\[
\mathcal{Z}_{r,g-r,m}(\sigma)\to \overline{\widetilde{\mathcal{A}}}_{g,m}
\]
lifting the composition
\[
\mathcal{Z}_{r,g-r,m}(\sigma)\to \mathcal{Z}_{r,g-r}(\sigma)\to \overline{\widetilde{\mathcal{A}}}_{g}.
\]
This extension necessarily factors through a finite map 
\[
\Gamma_m(\sigma)\backslash\mathcal{Z}_{r,g-r,m}\to \mathcal{Z}_m(r,\sigma),
\]
which is an isomorphism in the generic fiber. By looking at complete local rings, it is seen to be a finite \'etale map, and hence an isomorphism.

The last assertion about the formal completions now follows from~\cite[(A.3.2)]{mp:toroidal}. 
\end{proof}

From this, and the explicit nature of the map~\eqref{eqn:pushforward sigma}, we immediately obtain:
\begin{proposition}
\label{prop:divisor divisible}
Let $\eta_m\colon\overline{\widetilde{\mathcal{A}}}_{g,m}\to \overline{\widetilde{\mathcal{A}}}_{g}$ be the natural finite map, and let $\mathcal{D}_m\subset \overline{\widetilde{\mathcal{A}}}_{g,m}$ be the complement of $\widetilde{\mathcal{A}}_{g,m}$, equipped with its reduced scheme structure. 

Then $\mathcal{D}_m$ is a relative Cartier divisor over $\ZZ$. Moreover, if $\mathcal{D}\subset \overline{\widetilde{\mathcal{A}}}_g$ is the boundary divisor with its reduced scheme structure, then we have an equality of Cartier divisors $\eta_m^*\mathcal{D} = m\cdot \mathcal{D}_m$.
\end{proposition}

\begin{remark}
\label{rem:base change}
Note that the above proposition remains true if we replace $\widetilde{\mathcal{A}}_{g,m}$ and its compactification with the normalizations of their base change over $\mathcal{O}_K$, for any number field $K/\QQ$.
\end{remark}

\begin{proposition}
\label{prop:full level structure}
Let $\widetilde{\mathcal{A}}^{[m]}_g$ and $\overline{\widetilde{\mathcal{A}}}^{[m]}_g$ be as in \S~5. Let $\pi_m\colon \overline{\widetilde{\mathcal{A}}}_{g}^{[m]}\to \overline{\widetilde{\mathcal{A}}}_{g}$ be the natural finite map, and let $\mathcal{D}^{[m]}\subset \overline{\widetilde{\mathcal{A}}}_{g}^{[m]}$ be the complement of $\widetilde{\mathcal{A}}_{g}^{[m]}$, equipped with its reduced scheme structure. 

Then $\mathcal{D}^{[m]}$ is a relative effective Cartier divisor over $\ZZ$.\footnote{That is, it is an effective Cartier divisor that is flat over $\ZZ$.} Moreover, we have $\pi_m^*\mathcal{D} = m\cdot \mathcal{D}^{[m]}$.
\end{proposition}
\begin{proof}
Over $\ZZ[1/m,\mu_m]$, $\widetilde{\mathcal{A}}^{[m]}_g$ and $\overline{\widetilde{\mathcal{A}}}^{[m]}_g$ can be identified with a disjoint union of copies of $(\widetilde{\mathcal{A}}_{g,m})_{\ZZ[1/m,\mu_m]}$ and $(\overline{\widetilde{\mathcal{A}}}_{g,m})_{\ZZ[1/m,\mu_m]}$, respectively. So the result is true over $\ZZ[1/m]$. Moreover, by Proposition~\ref{prop:divisor divisible} and Remark~\ref{rem:base change}, it is true after a change of scalars to $\ZZ[\mu_m]$ followed by normalization. Combining the two, we find that the result is already true over $\ZZ$.
\end{proof}

\bibliographystyle{plain}             % (uses file "plain.bst")
\bibliography{levels}

\end{document}